\documentclass[11pt]{amsart}
\usepackage{amsfonts}
\usepackage{amssymb}
\usepackage[colorlinks=blue]{hyperref}
\usepackage{graphics}
\usepackage{hyperref}
\hypersetup{
colorlinks=true,
linkcolor=red,
citecolor=cyan,
}

\def\bbB{{\mathbb B}}
\def\bbC{{\mathbb C}}
\def\bbD{{\mathbb D}}

\def\cB{{\mathcal B}}

\def\cH{{\mathcal H}}

\newtheorem{theorem}{Theorem}[section]
\newtheorem{lemma}[theorem]{Lemma}

\usepackage{amsmath}
\theoremstyle{definition}
\newtheorem{definition}[theorem]{Definition}

\theoremstyle{remark}

\newtheorem*{rem}{Remark}
\numberwithin{equation}{section}

\begin{document}

\title{A note for Carleson measure on bounded \(\bbC\)-convex domains}

\author{Mingjin Li}
\address{School of Mathematical Sciences, Guizhou Normal University, Guiyang, 550025, P. R. China}
\email{limingjin2022@163.com}
\thanks{*Corresponding author.}

\author{Jianren Long*}
\address{School of Mathematical Sciences, Guizhou  Normal University, Guiyang, 550025, P. R. China.}
\email{longjianren2004@163.com}
\author{Lang Wang}
\address{School of Mathematical Sciences, Guizhou  Normal University, Guiyang, 550025,  P. R. China.}
\email{wanglang2020@amss.ac.cn}




\begin{abstract}
Let \(0<q<p<\infty\), \(\Omega\) be a bounded \(\bbC\)-convex domains in \(\bbC^n\). We establish several equivalent characterizations for the boundedness  of 
    Carleson embedding \(J_\mu:A_\alpha^p\hookrightarrow L^q(\mu)\) on \(\Omega\) with sharp \(\cB\) type condition.  Furthermore, we prove that the boundedness of \(J_\mu\) is equivalent to its compactness.
\end{abstract}

\maketitle

.

\section{Introduction and main result}
The notion of Carleson measure is an important object in complex analysis and function theory, which was introduced by Carleson \cite{Carleson1962} in his celebrated work on the corona problem in 1962. Specifically, Carleson characterized the boundedness of the Carleson embedding  \(J_\mu:H^p\hookrightarrow L^p(\mu)\)   by using Carleson square 
\(S_{\theta_0,h}=\{re^{i\theta}\in\bbD: 1-h\leq r<1, |\theta-\theta_0|\leq h\}\).  Later, Duren \cite{Duren} extended Carleson's result to the boundedness of  \(J_\mu:H^p\hookrightarrow L^q(\mu)\) for \(0<p\leq q<\infty\) by adapting the original argument, moreover,  the Carleson square  can be replaced by pseudohyperbolic disks.  Luecking \cite{Luecking1991} obtained the boundedness of \(J_\mu: H^p\to L^q(\mu)\) for the case of \(p>q\) by  using Khinchine's inequality and the theory of tent spaces associated with Hardy spaces.

The extension of \(q\)-Carleson measure to Bergman spaces has attracted considerable attention from many scholars, it was characterized by Hastings \cite{Hastings1975} and Oleinik \cite{Oleinik1978} for \(p\leq q\), Luecking \cite{Luecking1993} for \(p>q\).  In higher dimension, Cima-Wogen \cite{Cima1982} and Luecking-Duren-Weir \cite{Luecking1983,Duren2007} stuied the case of unit ball \(\bbB^n\) by using the Euclidean geometry of \(\bbB^n\) and conformal invariants (such as Bergman or Kobayashi balls) respectively. In order to extend the above results to other domains of \(\bbC^n\), Cima-Mercer and Abate-Raissy-Saracco  characterized the \(q\)-Carleson measure for \(A^p\)  on strongly  pseudoconvex domains by using the Euclidean geometric quantities and intrinsic Kobayashi geometry respectively, see \cite{Abate2011,Abate2012,Abate2019,Cima1995} for the details. For the case of weakly pseudoconvex domains, Li-Liu-Wang \cite{Li2024} characterized the \(2\)-Carleson measurable for \(A^2\) by using Kobayashi balls on  bounded convex domains with smooth boundary of finite type. Bi-Su-Zhang \cite{Bi2024} characterized the \(2\)-Carleson measure for \(A^2\) and \(A^p\) respectively on bounded \(\bbC\)-convex domains.

    Khanh-Liu-Thuc \cite{Khanh2019} introduced the  sharp \(\cB\) type condition (see Sect.\ref{sharpb} below) for Bergman kernel. Roughly speaking,  Bergman kernel satisfies the sharp \(\cB\) type condition ensures that the Bergman projection is bounded on \(L^p\) for \(1<p<\infty\).  More recently,  there is a series of work involved function spaces and operator theory related such domains. For example, Khanh-Liu-Thuc-Tien \cite{Khanh2019,Khanh2021}  characterized the boundedness and Schatten class of Bergman-Toeplitz operators for weighted spaces on weakly pseudoconvex domains with Bergman kernel satisfies sharp \(\cB\) type condition.   Zhang \cite{Zhang2025} characterized the \(q\)-Carleson measure for the weighted Bergman spaces \(A_\alpha^p\) on bounded \(\bbC\)-convex domains with sharp \(\cB\)  type condition, and what's more, Zhang's results extended the results in \cite{Bi2024} to the case of \(p\leq q\) and  the weighted Bergman spaces \(A_\alpha^p\), but for the case of \(p>q\), Zhang didn't characterize the full Carleson measure, Zhang's result is stated as follows.
\begin{theorem}\cite[Theorem 1.8]{Zhang2025}\label{Zhang's th}
Let \(\Omega\subset \bbC^n\) be a bounded \(\bbC\)-convex domain whose Bergman kernel satisfies sharp \(\cB\) type condition, \(\mu\) be a finite positive Borel measure on \(\Omega\). Assume that \(0<q<p<\infty\), if \(\mu\) satisfies one of the following three conditions 
\begin{itemize}
  \item [(i)] \(p_\Omega(z)^\frac{2\alpha q}{p} \widetilde{\mu}_t(z)\in L^\frac{p}{p-q} \) for some \(t>1\);
  \item [(ii)] \(\{p_\Omega(a_j)^{2+\frac{2(\alpha-1)q}{p}}\widehat{\mu}_r(a_j)\}_{j=1}^\infty\in l^\frac{p}{p-q}\) for some \(r\in (0,1)\) and \(r\)-lattice \(\{a_j\}_{j=1}^\infty\);
  \item [(iii)] \(p_\Omega(z)^\frac{2\alpha q}{p}\widehat{\mu}_R\in L^\frac{p}{p-q}\) for some \(R\in (0,1)\).
\end{itemize}
Then \(J_\mu:A_\alpha^p\hookrightarrow L^q(\mu) \) is bounded.
\end{theorem}
Motivated by Theorem \ref{Zhang's th},  we give the following main result in this paper. Theorem \ref{mainth} gives an equivalent characterization of \(q\)-Carleson measure for \(A_\alpha^p\), which generalizes the Theorem \ref{Zhang's th}. Moreover, we also show that \(J_\mu\) is bounded if and only if \(J_\mu\) is compact for the case of \(p>q\).
\begin{theorem}\label{mainth}
Let \(\Omega\subset \bbC^n\) be a bounded \(\bbC\)-convex domain whose Bergman kernel of sharp \(\cB\) type. Let \(\mu\) be a finite positive Borel measure, \(0<q<p<\infty\). Then the following statements are equivalent.
\begin{itemize}
  \item [(i)] \(J_\mu: A_\alpha^p\hookrightarrow L^q(\mu)\) is compact;
  \item [(ii)] \(J_\mu: A_\alpha^p\hookrightarrow L^q(\mu)\) is bounded;
  \item [(iii)] For some \(r\in (0,1)\) and \(r\)-lattice \(\{a_j\}_{j=1}^\infty\), we have 
  \[\left\{p_\Omega(a_j)^{2+\frac{2q(\alpha-1)}{p}}\widehat{\mu}_r(a_j)\right\}_{j=1}^\infty\in l^\frac{p}{p-q};\]
  \item [(iv)] \(p_\Omega(z)^\frac{2\alpha q}{p}\widehat{\mu}_R(z)\in L^\frac{p}{p-q}(\Omega)\) for some \(R\in (0,1)\). 
  \item [(v)] \(p_\Omega(z)^\frac{2\alpha q}{p} \widehat{\mu}_R(z)\in L^\frac{p}{p-q}(\Omega)\) for all \(R\in (0,1)\). 
\end{itemize}
\end{theorem}

This paper is organized as follows. Some preliminaries are given in Section \ref{section2}; An atomic type decomposition and the proof of Theorem \ref{mainth} are showed in Section \ref{section3}.

\section{Preliminaries} \label{section2}
\subsection{Notations} 
\begin{itemize}
  \item For a domain \(\Omega\),  \(\cH(\Omega)\) denotes the  set of holomorphic functions on \(\Omega\).
    \item  \(|\cdot|\) denote the standard Euclidean norm.
    \item For \(z\in \Omega\subsetneq \bbC^n\),  \(\delta_\Omega(z):=\inf\limits_{z_0\in\partial\Omega} |z-z_0|\).
     \item  \(a\lesssim b\) or \(a\gtrsim b\) means that there exists  a positive constant \(C\) such that \(a\leq Cb\) or \(a\geq Cb\), \(a\asymp b\) if \(a\lesssim b\) and \(a\gtrsim b\). 
\end{itemize}

\subsection{Invariants} \label{section2.2}

\hfil

 Given a domain \( \Omega\subset \mathbb{C}^n \) (with \( n \geq 2 \)), the (infinitesimal) \textit{Kobayashi metric} is the pseudo-Finsler metric defined by  
\[
\kappa_\Omega(x; v) := \inf \bigl\{ |\xi| : f \in H(\mathbb{D}, \Omega), \text{ with } f(0) = x,\; (df)_0(\xi) = v \bigr\},
\] where \(H(\bbD,\Omega)\) denotes the holomorphic section from unit disc \(\bbD\) to \(\Omega\).
For any curve \(\sigma : [a, b] \to \Omega\), the \textit{Kobayashi length} is
\[
l_\kappa(\sigma) := \int_a^b \kappa_\Omega \bigl( \sigma(t); \sigma'(t) \bigr) \, dt,
\]
 and the \textit{Kobayashi pseudo-distance} on \(\Omega\) is
\begin{align*}
  d_\Omega^K(x, y) := &\inf_{\sigma} \bigl\{ l_\kappa(\sigma) \mid \sigma : [a, b] \to \Omega \text{ is any absolutely continuous curve with } \\&\sigma(a) = x \text{ and } \sigma(b) = y \bigr\}.
\end{align*}



Suppose \( \Omega \) is a bounded domain in \(\bbC^n\), \( z \in \Omega \) and \( r \in (0, 1) \). Let \( B_\Omega(z, r) \) be the Kobayashi balls of centre \( z \) and radius \( \frac{1}{2} \log \frac{1 + r}{1 - r} \), i.e. 
\[
B_\Omega(z, r) = \bigl\{ z \in X \mid \tanh d_\Omega^K(z, w) < r \bigr\}.
\]
Note that \( \rho_\Omega = \tanh d_\Omega^K \) is still a distance on \( X \) since \(\tanh\) is a strictly convex function on \( \mathbb{R}^+ \).

For a bounded pseudoconvex domain \(\Omega\subset \bbC^n\),  Set \(A^2\) be the Hilbert spaces consists of all \(f\in \cH(\Omega)\) such that \(\|f\|_{A^2}^2=\int_\Omega |f(z)|^2 dV(z)<\infty\), where \(dV(z)\) is the ordinary Lebesgue measure. Riesz's theorem yields that there exists  a unique reproducing kernel function \(K_\Omega(z,w)\) such that 
\[f(z)=\int_\Omega f(w)\overline{K_\Omega(w,z)}dV(w)=\int_\Omega f(w)K_\Omega(z,w)dV(w),\] \(K_\Omega(z,w)\) is called the \textit{Bergman kernel}. For \(\alpha\leq 0\), the weighted Lebesgue space \(L_\alpha^p\) consists of all measurable functions such that \[\|f\|_{L_\alpha^p}^p=\int_\Omega |f(z)|^p K_\Omega(z,z)^\alpha dV(z)<\infty,\]  \(A_\alpha^p:=L_\alpha^p\bigcap \cH(\Omega)\).

Let \(X\) be the space consists of holomorphic functions, for \(0<p,q<\infty\) and a positive Borel measure \(\mu\), we say that \(\mu\) is a \(q\)-Carleson measure for \(X\) if there exists a finite constant \(C>0\) such that 
  \[\|f\|_{L^q(\mu)}\leq C\|f\|_X,\quad f\in X,\] where \(\|f\|_{L^q(\mu)}=\left(\int_\Omega |f(z)|^q d\mu\right)^\frac{1}{q}<\infty.\) It is clear that \(\mu\) is a \(q\)-Carleson measure for \(X\) if and only if the Carleson embedding
    \[J_\mu:X\hookrightarrow L^q(\mu)\] is bounded.
    What's more, \(\mu\) is a vanishing Carleson measure if the embedding \(J_\mu\) is compact in the sense that  \(\lim\limits_{j\to\infty}\|f_j\|_{L^q(\mu)}=0\)  whenever \(\{f_j\}_{j=1}^\infty\) is a bounded sequence in \(X\) which converges to \(0\) uniformly on any compact subset of \(\Omega\). For any given \(r\in(0,1)\), the averaging function of \(\mu\) related to Kobayashi ball is defined as 
    \[\widehat{\mu}_r(z):=\frac{\mu(B_\Omega(z,r))}{V(B_\Omega(z,r))},\quad z\in \Omega,\] where \(V(B_\Omega(z,r))\) is the Euclidean volume of \(B_\Omega(z,r)\).

\subsection{ Geometry of \(\bbC\)-convex domains}
\begin{definition}\label{defC-convex}
 Suppose \(\Omega\) is a domain in \(\bbC^n\). For any affine complex lines \(l\), if \(l\bigcap\Omega\) is simply connected whenever \(l\bigcap\Omega\neq\emptyset\), then \(\Omega\) is said to a \textit{\(\bbC\)-convex }domain.
\end{definition}

  The following geometric objects related to an arbitrary domain \(\Omega \subset \mathbb{C}^n\) containing no complex lines are needed in this paper. Let \(q \in \Omega\), and let \(\delta_{\Omega}(q)\) denote the distance from \(q\) to the boundary \(\partial \Omega\). Fix \(q_1 \in \partial \Omega\) so that 
\[
\eta_1(q) := |q_1 - q| = \delta_{\Omega}(q).
\] 
Put
\[
H_1 := q + \operatorname{Span}(q_1 - q)^\perp \quad \text{and} \quad \Omega_1 := \Omega \cap H_1.
\] 
Let \(q_2 \in \partial \Omega_1\) so that 
\[
\eta_2(q) := |q_2 - q| = d_{\Omega_1}(q).
\] 
Set 
\[
H_2 := q + \operatorname{Span}(q_1 - q, q_2 - q)^\perp, \quad \Omega_2 := \Omega \cap H_2,
\] 
and continuing this procedure we are led to an orthonormal basis 
\[
e_j = \frac{q_j - q}{|q_j - q|}, \quad 1 \leq j \leq n,
\] 
which is called \textit{minimal for \(\Omega\) at \(q\)}. It should be mentioned that after rotation we may always assume that \(e_1, e_2, \dots, e_n\) is the standard basis of \(\mathbb{C}^n\). Clearly, if \(\Omega\) is a bounded domain, then for any \(z \in \Omega\) we have
\[
\eta_1(z) \leq \eta_2(z) \leq \dots \leq \eta_n(z) \leq \operatorname{diam}(\Omega) < \infty,
\]
 where 
\(
\operatorname{diam}(\Omega) := \sup\{|z-w|: z, w \in \Omega\}.
\)


\newpage

The following lemma describes the geometric characteristics of Kobayashi balls on bounded \(\bbC\)-convex domains. 
\begin{lemma}\cite[Theorem 1]{NIkolv2011}\label{polydisks}
   Let $\Omega \subset \mathbb{C}^n$ be a $\mathbb{C}$-convex domain containing no complex lines, and $z \in \Omega$. Then for any $r \in (0, 1)$,
   \[
\mathbb{D}^n(z, c_1\eta(z)) \subset B_\Omega(z, r) \subset \mathbb{D}^n(z, c_2\eta(z)),
\]
where $c_i$ $(i = 1, 2)$ is some constant depending only on $n$ and $r$, $\eta(z) := (\eta_1(z), \dots, \eta_n(z))$, and
\[
\mathbb{D}^n(z, \eta(z)) := \{w \in \mathbb{C}^n : |w_i - z_i| < \eta_i(z),\ i = 1, \dots, n\}.\]
\end{lemma}

For plurisubharmonic functions, the sub mean value inequality are valid for Kobayashi balls on bounded \(\bbC\)-convex domains.

\begin{lemma}\cite[Proposition 2.3]{Bi2024}\label{submean}
Let $\Omega \subset \mathbb{C}^n$ be a bounded $\mathbb{C}$-convex domain. For any $r \in (0, 1)$, we denote $R := \frac{1}{2}(1 + r) $. Then for any $z_0 \in \Omega$ and $z \in B_{\Omega}(z_0, r)$, we have
\[
\varphi(z)\lesssim  \frac{1}{V(B_{\Omega}(z_0, r))} \int_{B_{\Omega}(z_0, R)} \varphi dV
\]
for all nonnegative plurisubharmonic functions $\varphi$ on $\Omega$. 
\end{lemma}

The following cover type lemma plays a significant role in the proof of Theorem \ref{mainth}.
\begin{lemma}\cite[Lemma 2.4]{Bi2024}\label{cover}
 Let \(\Omega \subset \mathbb{C}^n\) be a bounded \(\mathbb{C}\)-convex domain. For any \(r \in (0, 1)\), there exists a positive integer \(m\) and a sequence \(\{a_j\}_{j=0}^\infty \subset \Omega\) satisfy the following two conditions:
\begin{itemize}
  \item [(i)] \(\Omega = \bigcup_{j=0}^\infty B_\Omega(a_j, r)\);
  \item [(ii)] Every point in \(\Omega\) belongs to at most \(m\) of the balls \(B_\Omega(a_j,R)\), where \(R := \frac{1}{2}(1 + r)\).
\end{itemize}
  Such a sequence \(\{a_j\}_{j=0}^\infty \subset \Omega\) is said to be an \(r\)-lattice in \(\Omega\).
\end{lemma}
  The lower estimate for the Bergman kernel is given as follows.
 \begin{lemma}\cite[Lemma 2.4]{Zhang2025}\label{lower}
  Let \(\Omega \subset \mathbb{C}^n\) be a bounded \(\mathbb{C}\)-convex domain. Then for any \(r \in (0,1)\), \(z \in \Omega\) and \(w \in B_\Omega(z,r)\), we have 
\[
|K_\Omega(z,w)| \gtrsim K_\Omega(z,z)^{\frac{1}{2}} K_\Omega(w,w)^{\frac{1}{2}},
\]

where the constant in \(\gtrsim\) depends only on \(r\) and the domain \(\Omega\).
 \end{lemma}

For \(z\in \Omega\) and \(r\in (0,1)\). Let  \(p_\Omega(z)=\prod\limits_{i=1}^n\eta_i(z)\), then by \cite{Zhang2025}, for some positive constant \(C=C(r)\) we have 
\begin{equation}\label{jubukebi}
 \frac{1}{C} p_\Omega(z)\leq  p_\Omega(w)\leq Cp_\Omega(z),\quad w\in B_\Omega(z,r).
\end{equation}

By \eqref{jubukebi}, the following estimate for Bergman kernel  on diagonal is given.
\begin{lemma}\cite[Lemma 2.5]{Zhang2025}\label{diagonal}
   Let \(\Omega \subset \mathbb{C}^n\) be a \(\mathbb{C}\)-convex domain containing no complex lines. Then
\[
K_\Omega(z,z) \asymp \rho_\Omega(z)^{-2}, \quad z \in \Omega.\]
\end{lemma}

\begin{rem}
 By \eqref{jubukebi} and Lemma \ref{diagonal}, for \(w\in B_\Omega(z,r)\),
we have 
\begin{equation}\label{lower2}  
  |K_\Omega(z,w)|\gtrsim K_\Omega(z,z).
\end{equation} 
\end{rem}

\subsection{Bergman kernel with sharp \(\cB\) type conditions}\label{sharpb}
 
\hfil

In this subsection, we introduce a sharp \(\cB\) type condition for Bergman kernel, we firstly recall that the definition of {\(\cB\)- systems}.  

Let $\Omega$ be a bounded domain in $\mathbb{C}^n$. For $z \in \overline{\Omega}$ near the boundary $\partial \Omega$, a family of functions $\mathcal{B} = \{b_j(z,\cdot)\}_{j=1}^n$ is called a \textit{$\mathcal{B}$-system} at $z$ if there exist a neighborhood $U$ of $z$ and a positive integer $m \geq 2$ such that for all $w \in U$,
\[
b_1(z,w) := \frac{1}{\delta(z,w)} \quad \text{and} \quad b_j(z,w) := \sum_{k=2}^{m} \left( \frac{A_{jk}(z)}{\delta(z,w)} \right)^{\frac{1}{k}}, \quad  j = 2, \ldots, n,
\]
where $A_{jk}$ are non-negative bounded functions, and $\delta(z,w)$ is the pseudo-distance between $z$ and $w$ defined as follows:
\[
\delta(z,w) := |\delta_{\Omega}(z)| + |\delta_{\Omega}(w)| + |z_1 - w_1| + \sum_{l=2}^n \sum_{s=2}^m A_{ls}(z) |z_l - w_l|^s.
\]

For \(z\in\overline{\Omega}\) near the boundary, the Bergman kernel $K_\Omega$ is said to be of $\mathcal{B}$-type at $z \in \overline{\Omega}$ near the boundary $\partial \Omega$ if there exist positive constants $c$ and $C$ such that for all $w \in \Omega \cap \mathbb{B}(z, c)$,
\[
|K_\Omega(z, w)| \leq C \prod_{j=1}^n b_j^2(z, w).
\]
The Bergman kernel $K_\Omega$ is said to be of sharp $\mathcal{B}$ type at $z $ if $K_\Omega$ is of $\mathcal{B}$-type and has the sharp lower-bound on diagonal, i.e.
\[
K_\Omega(z, z) \asymp \prod_{j=1}^n b_j^2(z, z).\]


We finish this section with a Forelli-Rudin estimate on domains with Bergman kernel satisfying the sharp \(\cB\) type condition, which will be used to construct an atomic type decomposition.

\begin{lemma}\label{Rudin}\cite[Proposition 2.7]{Zhang2025}
  Let $\Omega \subset \mathbb{C}^n$ be a bounded domain whose Bergman kernel $K_\Omega$ is of sharp $\mathcal{B}$ type. Assume that $\alpha \leq 0, \, \beta \geq 0, \, t + \alpha \geq 1$ and $-1 < \beta - 2\alpha < 2t - 2$. Then the following inequality
\[
\int_{\Omega} |K_\Omega(z,w)|^t K_\Omega(w,w)^\alpha d_\Omega(w)^\beta dV(w) \lesssim K_\Omega(z,z)^{t+\alpha-1} d_\Omega(z)^\beta
\]
holds for every $z \in \Omega$.
\end{lemma}

\section{Proof of main result}\label{section3}
In this section, we give the proof of Theorem \ref{mainth}.
In order to prove theorem \ref{mainth}, we need to construct atomic type decomposition of \(A_\alpha^p\),  which is stated as follows.
\begin{theorem}\label{atomic}
  Let \(r\in(0,1)\) be fixed, and \(\{w_k\}\) be an \(r\)-lattice. For \(0<p<\infty\) and \(\{c_k\}\in l^p\), \(\alpha+Np>1\). Set \(f(z)= \sum\limits_{k=1}^\infty c_k K_\Omega(w_k,w_k)^{\frac{1-\alpha}{p}-N}K_\Omega(z,w_k)^N\), 
 then we have  \[\|f\|_{A_\alpha^p}\lesssim \|c_k\|_{l^p}\asymp 1.\]
\end{theorem}

The following lemma can be obtained by Lemma \ref{Rudin} and Schur's test, it can be used to construction of atomic decomposition in the proof of Theorem \ref{atomic}.
\begin{lemma}\label{lemindex} Let  \(\Omega\subset\bbC^n\) be a bounded pseudoconvex domain whose Bergman kernel \(K_\Omega(z,w)\) satisfies sharp \(\cB\) type condition. Assume that 
   \begin{equation}\label{index}
  \begin{cases}
    & q + p\sigma + \alpha \leq 0, \\
     & s + q + p\sigma + \alpha \geq 1 , \\
      & -1 < -2(q + p\sigma + \alpha) < 2s - 2 ,\\
       & t + p'\sigma \leq 0 ,\\ 
       & s + t + p'\sigma \geq 1, \\
        & -1 < -2(t + p'\sigma + \alpha) < 2s - 2 ,\\
        & q+s+t-1\leq 0.
  \end{cases}  
  \end{equation}
  Then the integral operator 
  \[f\mapsto T_{q,s,t}f(z)=K_\Omega(z,z)^q\int_\Omega K_\Omega(z,w)^s K_\Omega(w,w)^t f(w)dV(w)\] is bounded on \(L^p(dV_\alpha)\).  
\end{lemma}
\begin{proof}
  To prove the boundedness of \(T_{q,s,t}\), we only prove the boundedness of \(S_{q,s,t}f(z)=K_\Omega(z,z)^q\int_\Omega |K_\Omega(z,w)|^s K_\Omega(w,w)^t f(w)dV(w)\).
  To do this, let \(h(z)=K_\Omega(z,z)^\sigma\),
  \begin{align*}
    &K_\Omega(z,z)^q\int_\Omega |K_\Omega(z,w)|^s K_\Omega(w,w)^tdV(w)\\
    &=K_\Omega(z,z)^q\int_\Omega |K_\Omega(z,w)|^s K_\Omega(w,w)^{t-\alpha}dV_\alpha(w).
  \end{align*}
  On the one hand, by Lemma \ref{Rudin} and 
  \begin{align*}
   &K_\Omega(z,z)^q\int_\Omega |K_\Omega(z,w)|^s K_\Omega(w,w)^{t-\alpha}K_\Omega(w,w)^{p'\sigma}dV_\alpha(w)\\
   &=K_\Omega(z,z)^q\int_\Omega |K_\Omega(z,w)|^s K_\Omega(w,w)^{t+p'\sigma}dV(w)\\
   &\lesssim K_\Omega(z,z)^{q+s+t-1+p'\sigma}\\
   &\leq C K(z,z)^{p'\sigma}=Ch(z)^{p'}.
  \end{align*}
  On the other hand, 
  \begin{align*}
    &\int_\Omega K_\Omega(z,z)^q |K_\Omega(z,w)|^sK_\Omega(w,w)^{t-\alpha}K_\Omega(z,z)^{p\sigma}dV_\alpha(z)\\
    &=K_\Omega(w,w)^{t-\alpha}\int_\Omega |K_\Omega(z,w)|^sK_\Omega(z,z)^{q+\alpha+p\sigma}dV(z)\\
    &\lesssim K_\Omega(z,z)^{t+s+p\sigma-1}\\
    &\leq C K_\Omega(w,w)^{p\sigma}\\
    &=Ch(w)^p.
  \end{align*}
  Applying Schur's test, we obtain  that \(S_{q,s,t}\) is bounded on \(L^p(dV_\alpha)\), the proof is completed.
\end{proof}
Now we give the proof of Theorem \ref{atomic}.
\begin{proof}[Proof of Theorem \ref{atomic}]
  It is obvious that the \(f\) is holomorphic.
  For \(Np+\alpha>1\), by Lemma \ref{Rudin}, we have 
  \begin{equation}\label{norm}
  \|K_\Omega(z,w_k)^N\|_{A_\alpha^p}\lesssim K_\Omega(w_k,w_k)^{\frac{\alpha-1}{p}+N}.
  \end{equation}
  For \(0<p\leq 1\), using the fact that 
  \[\left(\sum_{j=1}^\infty b_j\right)^p\leq \sum_{k=1}^\infty b_j^p\,\,\, \text{for}\,\,\, b_j\geq 0,\]
  \eqref{norm} yields that 
  \[\|f\|_{A_\alpha^p}^p\lesssim \sum_{k=0}^\infty |c_k|^pK_\Omega(w_k,w_k)^{1-\alpha-pN}\|K_\Omega(\cdot, w_k)^N\|_{A_\alpha^p}^p\lesssim \sum_{k=1}^\infty |c_k|^p.\]
  For \(1<p<\infty\), set $$
  \begin{cases}
    &D_k= B_\Omega(w_k, r),\\
    &F(z)=\sum_{k=0}^\infty |c_k|v_\alpha(D_k)^{-\frac{1}{p}}\chi_k(z),
  \end{cases}
  $$
  then we have
  \[\|F\|_{A_\alpha^p}^p\lesssim \|c_k\|_{l^p}^p.\]

  Let \(q=1-s-t, s=N, t=1-N\) in  \eqref{index}, then by lemma \ref{polydisks}, lemma \ref{lower} and \eqref{jubukebi}, we have 
    \begin{align*}
   & S_{q,s,t}F(z)\\
   &=K_\Omega(z,z)^{1-N-(1-N)}\int_\Omega |K_\Omega(z,w)|^N K_\Omega(w,w)^{1-N} \left(\sum_{k=0}^\infty |c_k|v_\alpha(D_k)^{-\frac{1}{p}}\chi_k(w)\right)dv(w)\\
    &=\sum_{k=0}^\infty |c_k|v_\alpha(D_k)^{-\frac{1}{p}}\int_{D_k}|K_\Omega(z,w)|^NK_\Omega(w,w)^Ndv(w)\\
    &\gtrsim \sum_{k=0}^\infty |c_k|K_\Omega(w_k,w_k)^{\frac{1-\alpha}{p}}\int_{D_k} |K_\Omega(z,w)|^NK_\Omega(w,w)^{1-N}dv(w)\\
    &\gtrsim\sum_{k=0}^\infty |c_k|K_\Omega(w_k,w_k)^{\frac{1-\alpha}{p}+1-N}\int_{D_k} |K_\Omega(z,w)|^Ndv(w).
  \end{align*}
 By Lemma \ref{submean}, 
  \begin{align*}
    |K_z(w_k)|^N&\lesssim\frac{1}{V(D_k)}\int_{D_k}|K_\Omega(z,w)|^Ndv(w)\\
&\lesssim p_\Omega(w_k,w_k)^{-2}\int_{D_k}|K_\Omega(z,w)|^Ndv(w)\\
&\asymp K_\Omega(w_k,w_k)\int_{D_k}|K_\Omega(z,w)|^Ndv(w),
  \end{align*}
  then 
  \begin{align*}
    S_{q,s,t}F(z)\gtrsim \sum_{k=0}^\infty |c_k|K_\Omega(w_k,w_k)^{\frac{1-\alpha}{p}-N}\gtrsim |f(z)|.
  \end{align*}
  Hence, lemma \ref{lemindex}   yields that
  \[\|f\|_{\alpha}^p\lesssim \|S_{q,s,t}F\|_{L_\alpha^p}^p\lesssim\|F\|_{L_\alpha^p}^p\lesssim \|c_k\|_{l^p}^p.\]
  The proof is completed. 
\end{proof}

In order to prove (iv) \(\Longleftrightarrow\) (v) in Theorem \ref{mainth}, we need the following results.
\begin{lemma}\label{wuguan}\cite[Lemma 5.3]{Zhang2025} 
    Let \( 1 < p < \infty \) and \( s \in \mathbb{R} \). Then the function \( p_\Omega^s \widehat{\mu}_r \in L^p(\Omega) \) for some \( r \in (0, 1) \) if and only if \( p_\Omega^s \widehat{\mu}_R \in L^p(\Omega) \) for all \( R \in (0, 1) \). Moreover, we have
    \[
    \| p_\Omega^s \widehat{\mu}_r \|_p \asymp \| p_\Omega^s \widehat{\mu}_R \|_p,
    \]
    where the constants in \( \asymp \) depend on \( r, R \in (0, 1) \).
\end{lemma}

\begin{lemma}\cite[Lemma 5.1]{Zhang2025}\label{lem4.2}
  Suppose \( r \in (0, 1) \) and \( 0 < p < \infty \). Then we have
\[\int_{\Omega} |f(z)|^p d\mu(z) \lesssim \int_{\Omega} |f(z)|^p \widehat{\mu}_r(z) dV(z)\]
for all \( f \in \cH(\Omega) \).
\end{lemma}

Now we give the proof of Theorem \ref{mainth}
\begin{proof}[Proof of Theorem \ref{mainth}]
  Thanks to the Montel's theorem and Lemma \ref{wuguan}, we only need to prove (ii) \(\implies\) (iii) \(\implies\) (iv) \(\implies\) (ii), (ii) \(\implies \) (i).

(ii) \(\implies\) (iii). For \(\{c_k\}\in l^p\), let
\[f(z)=\sum_{k=0}^\infty c_kK_\Omega(z,w_k)^NK_\Omega(w_k,w_k)^{\frac{1-\alpha}{p}-N}.\] 
By theorem \ref{atomic}, then \(f\in A_\alpha^p\) and \(\|f\|_{A_\alpha^p}^p\lesssim \sum_{k=0}^\infty |c_k|^p\).
Since \(J_\mu\) is bounded, hence 
  \begin{align}\label{bounded}
    \begin{aligned}
     &\int_{\Omega} \bigg|\sum_{k=0}^\infty c_kK_\Omega(z,w_k)^NK_\Omega(w_k,w_k)^{\frac{1-\alpha}{p}-N}\bigg|^qd\mu(z)\\
  &\leq\|\mu\|^q\|f\|_{A_\alpha^p}^q\\
  &\lesssim \|\mu\|^q\left(\sum_{k=0}^\infty |c_k|^p\right)^{\frac{q}{p}}. 
    \end{aligned}
  \end{align}

Let \[\begin{cases}
  r_0(t)=\begin{cases}
    1,\,\,0\leq t-[t]\frac{1}{2},\\
    -1,\,\,\frac{1}{2}\leq t-[t]<1,
  \end{cases}\\
  r_k(t)=r_0(2^kt), \,\,k=1,2,3,\dots
\end{cases}
\]
 be Rademacher functions on \((0,1)\) and set
 \[f_t(z)=\sum_{k=0}^\infty r_k(t) c_kK_\Omega(z,w_k)^NK_\Omega(w_k,w_k)^{\frac{1-\alpha}{p}-N}.\] Replacing \(f(z)\) by \(f_t(z)\) in \eqref{bounded}, integrate with respect to \(t\) from \(0\) to \(1\),  by Fubini's theorem and Khinchine's inequality, we obtain that 
\begin{align}\label{eq4.2}
  \begin{aligned}
  &\int_\Omega \left(\sum_{k=0}^\infty |c_k|^2K_\Omega(w_k,w_k)^{2(\frac{1-\alpha}{p}-N)}|K_\Omega(z,w_k)|^{2N}\right)^{\frac{q}{2}}d\mu(z)\\
  &\lesssim\int_\Omega\int_0^1 \bigg|\sum_{k=0}^\infty r_k(t) c_kK_\Omega(z,w_k)^NK_\Omega(w_k,w_k)^{\frac{1-\alpha}{p}-N}\bigg|^qdtd\mu(z)\\
 & =\int_0^1\int_\Omega \bigg|\sum_{k=0}^\infty r_k(t) c_kK_\Omega(z,w_k)^NK_\Omega(w_k,w_k)^{\frac{1-\alpha}{p}-N}\bigg|^qd\mu(z)dt\\
 &\lesssim\|\mu\|^q\left(\sum_{k=0}^\infty |c_k|^p\right)^{\frac{q}{p}}.
  \end{aligned}
\end{align}

On the one hand, if \(q\geq 2\), then
\begin{align*}
 & \int_\Omega \left(\sum_{k=0}^\infty |c_k|^2K_\Omega(w_k,w_k)^{2(\frac{1-\alpha}{p}-N)}|K_\Omega(z,w_k)|^{2N}\right)^{\frac{q}{2}}d\mu(z)\\
  &\geq \int_\Omega \left(\sum_{k=0}^\infty |c_k|^2K_\Omega(w_k,w_k)^{2(\frac{1-\alpha}{p}-N)}|K_\Omega(z,w_k)|^{2N}\chi_k(z)\right)^{\frac{q}{2}}d\mu(z)\\
  &\geq\int_{D_k} \sum_{k=0}^\infty |c_k|^qK_\Omega(w_k,w_k)^{q(\frac{1-\alpha}{p}-N)}|K_\Omega(z,w_k)|^{qN}d\mu(z).
\end{align*}

On the other hand, if \(0<q<2\), then H\"older's inequality yields that 
\begin{align*}
 & \int_{D_k} \sum_{k=0}^\infty |c_k|^qK_\Omega(w_k,w_k)^{q(\frac{1-\alpha}{p}-N)}|K_\Omega(z,w_k)|^{qN}d\mu(z)\\
&=\int_\Omega \sum_{k=0}^\infty |c_k|^qK_\Omega(w_k,w_k)^{q(\frac{1-\alpha}{p}-N)}|K_\Omega(z,w_k)|^{qN}\chi_k(z)d\mu(z)\\
&\leq\int_\Omega\left(\sum_{k=0}^\infty |c_k|^2K_\Omega(w_k,w_k)^{2(\frac{1-\alpha}{p}-N)}|K_\Omega(z,w_k)|^{2N}\right)^{\frac{q}{2}}\left(\sum_{k=0}^\infty \chi_k(z)\right)^{1-\frac{q}{2}}d\mu(z)\\
&\leq N^{1-\frac{q}{2}}\int_\Omega\left(\sum_{k=0}^\infty |c_k|^2K_\Omega(w_k,w_k)^{2(\frac{1-\alpha}{p}-N)}|K_\Omega(z,w_k)|^{2N}\right)^{\frac{q}{2}}d\mu(z).
\end{align*}

In a word 
\begin{align*}
  &\int_{D_k} \sum_{k=0}^\infty |c_k|^qK_\Omega(w_k,w_k)^{q(\frac{1-\alpha}{p}-N)}|K_\Omega(z,w_k)|^{qN}d\mu(z)\\
  &\lesssim\left(\sum_{k=0}^\infty |c_k|^2K_\Omega(w_k,w_k)^{2(\frac{1-\alpha}{p}-N)}|K_\Omega(z,w_k)|^{2N}\right)^{\frac{q}{2}}d\mu(z)
\end{align*}
holds for all \(0<q<\infty\). By \eqref{lower2} and \eqref{eq4.2}, we obtain that 
\begin{align}\label{dual}
  \sum_{k=0}^\infty|c_k|^qK_\Omega(w_k,w_k)^{q(\frac{1-\alpha}{p})}\mu(D_k)\lesssim\|\mu\|^q\left(\sum_{k=0}^\infty |c_k|^p\right)^\frac{q}{p}.
\end{align}

Let \(a_k=|c_k|^q\), then \(\{a_k\}\in l^\frac{p}{q}\). The fact that \((l^\frac{p}{q})^*=l^\frac{p}{p-q}\), Lemma \ref{diagonal} and \eqref{dual} yield
\begin{align*}
  \|\mu\|&\gtrsim \|\mu(D_k)K_\Omega(w_k,w_k)^{q(\frac{1-\alpha}{p})}\|_{l^\frac{p}{p-q}}^\frac{1}{q}\\
  &\asymp\|\widehat{\mu}_r(B_\Omega(w_k,r))p_\Omega(w_k)^{2+\frac{2q(\alpha-1)}{p}}\|_{l^\frac{p}{p-q}}^\frac{1}{q},
\end{align*}
which give the proof of (ii) \(\implies\) (iii)  if replace \(w_k\) by some \(r\)-lattice \(\{a_j\}\).

(iii)\(\implies\) (iv) have been proved in \cite{Zhang2025}, we rewrite it here for the sake of completeness.  

Let \(\{a_j\}_{j=1}^\infty\) be some \(r\)-lattice such that \(B_\Omega(z,r)\subset B_\Omega(a_j,r)\) for all \(z\in B_\Omega(a_j,r)\). For all \(z\in B_\Omega(a_j,r)\),  Lemma \ref{polydisks} and \eqref{jubukebi} yield that 
\begin{align}\label{sanjiaobudengshi}
  \begin{aligned}
\widehat{\mu}_{R}(z) &\approx p_{\Omega}(z)^{-2} \mu \left( B_{\Omega}(z, R) \right) \\
&\lesssim p_{\Omega}(a_{j})^{-2} \mu \left( B_{\Omega}(a_{j}, r) \right) \approx \widehat{\mu}_{r}(a_{j}).
  \end{aligned}
\end{align} 
Lemma \ref{polydisks}, lemma \ref{cover} and \eqref{sanjiaobudengshi} yield that 
\begin{align*}
\left\| p_{\Omega}(z)^{\frac{2\alpha q}{p}} \widehat{\mu}_R(z) \right\|_{p-q}^{\frac{p}{p-q}} 
&\lesssim \sum_{j=1}^{\infty} \int_{B_\Omega(a_j,r)} p_{\Omega}(z)^{\frac{2\alpha q}{p-q}} \widehat{\mu}_R(z)^{\frac{p}{p-q}} \, dV(z) \\
&\lesssim \sum_{j=1}^{\infty} p_{\Omega}(a_j)^{2+\frac{2\alpha q}{p-q}} \widehat{\mu}_r(a_j)^{\frac{p}{p-q}} \\
&= \left\| \left\{ p_{\Omega}(a_j)^{2+\frac{2(\alpha-1)q}{p}} \widehat{\mu}_r(a_j) \right\}_{j=1}^{\infty} \right\|_{l^{\frac{p}{p-q}}}^{\frac{p}{p-q}} < \infty.
\end{align*}


(iv) \(\implies\) (ii). Since \(p>q\), hence \(\frac{1}{\frac{p}{q}}+\frac{1}{\frac{p}{p-q}}=1\). For \(f\in A_\alpha^p\), Lemma \ref{lem4.2} and Lemma \ref{diagonal} yield that 
\begin{align*}
 & \int_{\Omega} |f(z)|^q d\mu(z) \\
 &\lesssim \int_{\Omega} |f(z)|^q \widehat{\mu}_R(z) dV(z)\\
&\leq \bigg\|| f(z)|^q K_\Omega(z, z)^{\frac{\alpha q}{p}} \Big\|_{\frac{p}{q}} \cdot \left\| K_\Omega(z, z)^{-\frac{\alpha q}{p}} \widehat{\mu}_R(z) \right\|_{\frac{p}{p-q}}\\
&\asymp\left\| p_\Omega(z)^{\frac{2\alpha q}{p}} \widehat{\mu}_R(z) \right\|_{\frac{p}{p-q}}  \| f \|_{A_\alpha^p}^q\\
&\lesssim  \| f \|_{A_\alpha^p}^q. 
\end{align*}


(ii) \(\implies\) (i).  Let \(\{f_j\}\) be any bounded sequence in \(A_\alpha^p\) and \(f_j\to 0\) uniformly on each compact subset of \(\Omega\) as \(j\to \infty\). Since 
\(p_\Omega(z)^\frac{2\alpha q}{p}\widehat{\mu}_R(z)\in L^\frac{p}{p-q}\), hence there exists a set \(K_R\subset\Omega\) such that 
\begin{equation}\label{compact}
  \int_{\Omega\setminus K_R}\bigg|p_\Omega(z)^\frac{2\alpha q}{p}\widehat{\mu}_R(z)\bigg|^\frac{p}{p-q}dV(z)<\varepsilon.
\end{equation}

Consider 
\begin{align*}
  \int_\Omega |f_j(z)|^qd\mu(z)&=\left(\int_{K_R}+\int_{\Omega\setminus K_R}\right)|f_j(z)|^qd\mu(z)\\
  &=I_1+I_2.
\end{align*}
For \(I_1\), since \(f_j\to 0\) uniformly on each compact subset of \(\Omega\), then we have 
\[I_1=\int_{K_R} |f_j(z)|^q\leq C\sup_{z\in K_R} |f_j(z)|^q<\frac{\varepsilon}{2},\quad j\to\infty.\]

Let \(\mu_R\) be the restriction of \(\mu\) to \(\Omega\setminus K_R\), noting that  \(J_\mu\) is bounded if and only if \(p_\Omega(z)^\frac{2\alpha q}{p}\widehat{\mu}_R\in L^\frac{p}{p-q}\), by \eqref{compact} we obtain that 
\begin{align*} 
  I_2=&\int_{\Omega\setminus K_R}|f_j(z)|^qd\mu(z)\\
  &=\int_\Omega|f_j(z)|^qd\mu(z)\\
  &\lesssim\left\| p_\Omega(z)^{\frac{2\alpha q}{p}} \widehat{\mu}_R(z) \right\|_{\frac{p}{p-q}}  \| f \|_{A_\alpha^p}^q\\
  &\lesssim \frac{\varepsilon}{2}.
\end{align*}
Therefore, 
\[\lim_{j\to\infty}\int_\Omega |f_j(z)|^qd\mu(z)=0,\]
which means that \(\mu\) is a vanishing \(q\)-Carleson measure, hence \(J_\mu\) is compact.
The proof is completed. 
\end{proof} 

\section*{Acknowledgments} This research work was supported by the National Natural Science Foundation of China~(Grant No.~12261023,~11861023.)

\bibliographystyle{amsplain}

\end{document}